\documentclass{amsart}
\usepackage{amsmath,amsfonts,amsthm,amssymb,verbatim,enumerate,quotes}
\newcommand{\comments}[1]{}
\newtheorem{theo}{Theorem}[section]

\newtheorem{lemma}{Lemma}[section]
\newtheorem{corollary}{Corollary}[section]

\newtheorem*{remark}{Remark}
\newtheorem{remarkk}{Remark}
\numberwithin{equation}{section}
\def \isnatural {\in\mathbb{N}}

\def \iscomplex {\in\mathbb{C}}

\def \nhd {neighbourhood}
\def \usc {upper semicontinuous}
\def \lsc {lower semicontinuous}
\def \spw {spider's web}
\def \spws {spiders' webs}
\def\rom{\textup}
\newcommand{\intb} {\partial_{\text{\rmfamily\rom{int}\normalfont}}}
\newcommand{\innb} {\partial_{\text{\rmfamily\rom{inn}\normalfont}}}
\newcommand{\outb} {\partial_{\text{\rmfamily\rom{out}\normalfont}}}
\newcommand{\tef}{transcendental entire function}
\newcommand\qfor{\quad\text{for }}
\begin{document}
%
%
\title[On fundamental loops and the fast escaping set]{On fundamental loops and the fast escaping set}
\author{D. J. Sixsmith}
\address{Department of Mathematics and Statistics \\
	 The Open University \\
   Walton Hall\\
   Milton Keynes MK7 6AA\\
   UK}
\email{david.sixsmith@open.ac.uk}
%
%
\begin{abstract}
%
The fast escaping set, $A(f)$, of a {\tef} $f$ has begun to play a key role in transcendental dynamics. In many cases $A(f)$ has the structure of a {\spw}, which contains a sequence of fundamental loops. We investigate the structure of these fundamental loops for functions with a multiply connected Fatou component, and show that there exist {\tef}s for which some fundamental loops are analytic curves and approximately circles, while others are geometrically highly distorted. We do this by introducing a real-valued function which measures the rate of escape of points in $A(f)$, and show that this function has a number of interesting properties.
\end{abstract}
\maketitle
%
%
%
\section{Introduction}
Suppose that $f:\mathbb{C}\rightarrow\mathbb{C}$ is a {\tef}. The \itshape Fatou set \normalfont $F(f)$ is defined as the set of points $z\iscomplex$ such that $(f^n)_{n\isnatural}$ is a normal family in a {\nhd} of $z$. The \itshape Julia set \normalfont $J(f)$ is the complement in $\mathbb{C}$ of $F(f)$. An introduction to the properties of these sets was given in \cite{MR1216719}.

For a general {\tef} the \itshape escaping set \normalfont $$I(f) = \{z : f^n(z)\rightarrow\infty\text{ as }n\rightarrow\infty\}$$ was studied first in \cite{MR1102727}. This paper concerns a subset of the escaping set, called the \itshape fast escaping set \normalfont $A(f)$. This was introduced in \cite{MR1684251}, and can be defined \cite{Rippon01102012} by
\begin{equation}
\label{Adef}
A(f) = \{z : \text{there exists } \ell \isnatural \text{ such that } |f^{n+\ell}(z)| \geq M^n(R,f), \text{ for } n \isnatural\}.
\end{equation}
Here, the \itshape maximum modulus function \normalfont $M(r,f) = \max_{|z|=r} |f(z)|,$ for $r \geq 0,$ $M^n(r,f)$ denotes repeated iteration of $M(r,f)$ with respect to the variable $r$, and $R > 0$ is such that $M^n(R,f)\rightarrow\infty$ as $n\rightarrow\infty$. For simplicity, we only write down this restriction on $R$ in formal statements of results -- elsewhere this should be assumed to be true. We write $M(r)$ when it is clear which function is being referred to.

In \cite{Rippon01102012} several results on $A(f)$ were proved by considering the closed sets
\begin{equation*}
A_R(f) = \{z : |f^n(z)| \geq M^n(R), \ n\isnatural\},
\end{equation*}
where $R > 0$ is such that $M^n(R)\rightarrow\infty$ as $n\rightarrow\infty$. Rippon and Stallard \cite{Rippon01102012} used properties of $A_R(f)$ and $A(f)$ to develop new results relating to Eremenko's conjecture \cite{MR1102727} that $I(f)$ contains no bounded components. In addition they introduced the concept of a \itshape {\spw}\normalfont. A set $E$ is a {\spw} if $E$ is connected and there exists a sequence of bounded simply connected domains $(G_n)_{n\isnatural}$ such that $$\partial G_n \subset E, \ G_n \subset G_{n+1}, \text{ for } n\isnatural, \text{ and } \bigcup_{n\isnatural}~G_n~=~\mathbb{C}.$$ 

A {\tef} for which $A_R(f)$ is a {\spw} has very strong dynamical properties -- for example, it is shown in \cite{Rippon01102012} that $A(f)$ and $I(f)$ are also {\spws}, Eremenko's conjecture holds and all components of the Fatou set are bounded. There are many large classes of {\tef}s for which it is known that $A_R(f)$ is a {\spw}; see \cite{Rippon01102012}, \cite{MR2917092} and \cite{MR2838342} for examples.

To understand the structure of $A_R(f)$ {\spws}, Rippon and Stallard \cite{Rippon01102012} introduced \itshape fundamental holes \normalfont and \itshape fundamental loops\normalfont. When $A_R(f)$ is a {\spw}, we define the fundamental hole $H_R$ as the component of $A_R(f)^c$ that contains the origin, and the fundamental loop $L_R$ by $L_R = \partial H_R$. Since $A_R(f)$ is closed, we have that $L_R \subset A_R(f)$.

Our notation here differs slightly from that in \cite{Rippon01102012}. For $R>0$ fixed, Rippon and Stallard define sets $$A_R^m(f) = \{z : |f^n(z)| \geq M^{n+m}(R), \ n\isnatural, \ n+m\geq 0\},\qfor m\geq 0,$$ and define the sequence of fundamental holes to be the components of $A_R^m(f)^c$ that contain the origin. Denoting this sequence by $(H_m')_{m\geq 0}$, we observe that these notations are related by the equation $$H_{M^m(R)} = {H_m'}, \qfor m\geq 0.$$ 

It was shown in \cite[Theorem 1.9 (a)]{Rippon01102012} that $A_R(f)$ is a {\spw} whenever $f$ is a {\tef} with a multiply connected Fatou component. 
We now give the first results on the properties of fundamental loops in this case. The first of these gives information on the location of some fundamental loops. We say that a set $U$ surrounds a set $V$ if and only if $V$ is contained in a bounded component of $\mathbb{C}\backslash U$. We also write \rmfamily dist\normalfont$(z, U) = \inf_{w\in U} |z-w|$.
\begin{theo}
\label{Tloops}
Suppose that $f$ is a {\tef}. Then there exists $R'=R'(f)>0$ such that the following holds. If $U$ is a multiply connected Fatou component of $f$, such that $U$ surrounds the origin and \rmfamily\upshape dist\itshape$(0, U) > R'$, then there exist $0 < R_1 < R_2$ such that
\begin{enumerate}[(a)]
\item $L_{R_1} ={\intb} U$; 
\item $L_{R_2} =\outb U$; 
\item if $L_R$ is a fundamental loop such that $L_R \cap U \ne \emptyset$, then $L_R\subset U$. Moreover, this condition occurs if and only if $R_1 < R < R_2$.
\end{enumerate}
\end{theo}
Here $\outb U$ is defined as the boundary of the unbounded component of $\mathbb{C}\backslash U$, and $\intb U$ is defined as the boundary of the component of $\mathbb{C}\backslash\overline{U}$ that contains the origin. The related set $\innb U$ is defined in \cite{2011arXiv1109.1794B} as the boundary of the component of $\mathbb{C}\backslash U$ that contains the origin.

In general, if $U$ is a Fatou component, we write $U_n$, $n\geq 0$, for the Fatou component containing $f^n(U)$. Note that, by Lemma \ref{Lbaker} below, if $V$ is a multiply connected Fatou component then there is an $N\isnatural$ such that, for $n\geq N$, $V_n$ is a multiply connected Fatou component which satisfies the hypotheses of Theorem~\ref{Tloops}.

%
%
Using Theorem~\ref{Tloops} we prove the following result.
\begin{theo}
\label{TloopsinJF}
Suppose that $f$ is a {\tef} and that $L_R$ is a fundamental loop of $f$. Then either $L_R \subset F(f)$ or $L_R \subset J(f)$.
\end{theo}
A second consequence of Theorem~\ref{Tloops} is that when a fundamental loop lies within a multiply connected Fatou component, $U$, it is often possible to say more about the nature of this set. In fact, there is a close relationship between some fundamental loops of $f$ and some level sets of the non-constant positive harmonic function $h$ that was introduced by Bergweiler, Rippon and Stallard in \cite[Theorem 1.2]{2011arXiv1109.1794B}, and used to prove many geometric properties of multiply connected Fatou components. The function $h$ is defined by
\begin{equation}
\label{hdefinU}
h(z) = \lim_{n\rightarrow\infty} \frac{\log|f^n(z)|}{\log|f^n(z_0)|}, \qfor z\in U, \text{ some } z_0 \in U.
\end{equation}

Our result is as follows.
\begin{theo}
\label{Tloopsandh}
Suppose that $f$ is a {\tef}. Then there exists $R'=R'(f)>0$ such that the following holds. If $U$ is a multiply connected Fatou component of $f$, such that $U$ surrounds the origin, \rmfamily\upshape dist\itshape$(0, U) > R'$, and $h$ is as defined as in (\ref{hdefinU}), then
\begin{enumerate}[(a)]
\item if $L_R\subset U$ is a fundamental loop, then $h(z)$ is constant on $L_R$ and so $L_R$ is an analytic Jordan curve; 
\item if $\Gamma$ is a level set of $h$, then $\Gamma$ has a component $\gamma$ which surrounds the origin and there is a fundamental loop $L_R$ such that $L_R\subset\gamma$.
\end{enumerate}
\end{theo}

It follows from these results that the fundamental loops of a {\tef} can have very varied geometrical properties. For example, consider the {\tef} $f$ given in  \cite[Example 3]{2011arXiv1109.1794B}. This has a multiply connected Fatou component $U$ with the property that $$\lim_{n\rightarrow\infty} \frac{\max \{\log |z| : {z\in\outb U_n}\}}{\min \{\log |z| : {z\in\outb U_n}\}} = \infty.$$ 
By Theorem~\ref{Tloops}, there is a fundamental loop of $f$ which coincides with $\outb U_n$, and so is far from circular for large values of $n$. However, there are also fundamental loops of $f$ which lie inside $U_n$, for each $n\isnatural$. By Theorem~\ref{Tloopsandh}(a) these are analytic Jordan curves, and by \cite[Theorem 7.1]{2011arXiv1109.1794B} can be approximately circular.

%
%
%
A key tool in the proofs of these theorems is a function $R_A$, defined in (\ref{RAdef}) below, which for a point $z$ is the largest $R$ such that $z \in A_R(f)$.  In general this function can only be defined in a subset of $A(f)$. In Section~\ref{SRA} we show that, subject to a certain normalisation, this definition can in fact be extended in a natural way to the whole complex plane. We show that, in this case, there is an alternative characterisation of $A(f)$. We also show that the function $R_A$ has a number of interesting properties. \\

%
%
The structure of this paper is as follows. First, in Section~\ref{Sbackground}, we state a number of results required in the proof of our main theorems. With the exception of Lemma~\ref{mconnlemma}, these are all known results. In Section~\ref{Sblaschke} we prove a new result, which states that if a {\tef} has a certain property with respect to a nested sequence of bounded simply connected domains, then there is a fixed point which has a certain `attracting' property. This may be of independent interest. In Section~\ref{SRAinMC} we show that the function $R_A$ can be defined in certain multiply connected Fatou components, and prove several preparatory lemmas. In Section~\ref{Sloops} we prove Theorems~\ref{Tloops}, \ref{TloopsinJF} and \ref{Tloopsandh}. Finally, in Section~\ref{SRA} we state and prove several results regarding the case when $R_A$ can be defined in the whole complex plane.
%
%
%
%
\section{Background material}
\label{Sbackground}
We use the following notation for an annulus and a disc $$A(r_1, r_2) = \{ z : r_1 < |z| < r_2 \}, \qfor 0 < r_1 < r_2, $$ $$B(\zeta, \ r) = \{ z : |z-\zeta| < r \}, \qfor 0 < r.$$

%
%
We require the following two well known facts about the maximum modulus of a {\tef}:
\begin{equation}
\label{Meq0}
\log M(e^t) \text{ is a convex and increasing function of } t,
\end{equation}
and
\begin{equation}
\label{Meq1}
\frac{\log M(r)}{\log r} \rightarrow\infty \text{ as } r\rightarrow\infty.
\end{equation}

We often use the following lemma \cite[Theorem 2.2]{2011arXiv1109.1794B}, generally with $n=1$.
\begin{lemma}
\label{LMtheo}
Let $f$ be a {\tef}. Then there exists $R_0~=~R_0(f)~>~0$ such that, for all $0 < c' < 1 < c$, and all $n\isnatural$, 
\begin{equation}
\label{Mbigeq}
M(r^c, f^n) \geq M(r, f^n)^c, \qfor r>R_0,
\end{equation}
and
\begin{equation}
\label{Msmalleq}
M(r^{c'}, f^n) \leq M(r, f^n)^{c'}, \qfor r>R_0^{1/c'}.
\end{equation}
\end{lemma}

%
%
We denote the inverse function of $M$, when this is defined, by $M^{-1}$. For simplicity we write $M^{-n}$, for $n\isnatural$, to denote $n$ repeated iterations of $M^{-1}$. Observe that $M^{-1}(r)$ is defined for $r~\in~[|f(0)|,~\infty)$ and is strictly increasing. Moreover, by (\ref{Meq0}) and \cite[Theorem 7.2.2]{MR2467621}, $\log M^{-1}(e^s)$ is a concave and increasing function of $s$. Also, if $R_0$ is the constant from Lemma~\ref{LMtheo}, then it follows from (\ref{Mbigeq}) that
\begin{equation}
\label{mu1}
M^{-1}(r^c) \leq M^{-1}(r)^c, \qfor r>\max\{M(R_0), \ |f(0)|\}, \ c > 1.
\end{equation}

%
%
We next require a number of results concerning multiply connected Fatou components. Our first is the following well-known result of Baker \cite[Theorem 3.1]{MR759304}, which we often use without comment.
\begin{lemma}
\label{Lbaker}
Suppose that $f$ is a {\tef} and that $U$ is a multiply connected Fatou component of $f$. Then each $U_n$ is bounded and multiply connected, $U_{n+1}$ surrounds $U_n$ for large $n$, and $U_n\rightarrow\infty$ as $n\rightarrow\infty$.
\end{lemma}
In particular, if $U$ is a multiply connected Fatou component, then $U$ is bounded and so $f^n : U \to U_n$ is a proper map, for $n\isnatural$; see \cite[Corollary 1]{MR1642181}. Hence $U_n = f^n(U)$, for $n\isnatural$.

%
%
We need a number of results from \cite{2011arXiv1109.1794B}. Suppose that $f$ is a {\tef} with a multiply connected Fatou component $U=U_0$, and let $z_0\in U$ be fixed. It follows from \cite[Theorem 1.2]{2011arXiv1109.1794B} that there exists $\alpha > 0$ such that, for large $n$, the maximum annulus centred at the origin, contained in $U_n$ and containing $f^n(z_0)$ is of the form
\begin{equation}
\label{Bneq}
B_n = A(r_n^{a_n},r_n^{b_n}), \ \text{where } r_n = |f^n(z_0)|, \ 0<a_n< 1-\alpha < 1+ \alpha < b_n.
\end{equation}
We require part of \cite[Theorem 1.5]{2011arXiv1109.1794B}.
\begin{lemma}
\label{LThe1.5}
Suppose that $f$ is a {\tef} with a multiply connected Fatou component $U$, and let $z_0 \in U$. For large $n\isnatural$, let $r_n, \ a_n$ and $b_n$ be as defined in (\ref{Bneq}), and let ${\underline{a}_n}$ denote the smallest value such that $$\{z : |z| = r_n^{\underline{a}_n}\} \cap \innb U_n \ne \emptyset.$$ Then, as $n \rightarrow\infty$, $$a_n\rightarrow a \in [0,1), \quad {\underline{a}_n}\rightarrow a, \quad\text{and}\quad {b_n}\rightarrow b\in (1, \infty].$$
\end{lemma}
We also need the following \cite[Theorem 1.3]{2011arXiv1109.1794B} which shows that any compact subset of $U$ eventually iterates into the maximal annulus $B_n$.
\begin{lemma}
\label{L1.3}
Let $f, \ U, \ z_0$ be as in Lemma~\ref{LThe1.5}. For large $n\isnatural$, let $r_n, \ a_n, \ b_n$ and $B_n$ be as in (\ref{Bneq}). Then, for each compact set $C \subset U$, there exists $N\isnatural$ such that
\begin{equation}
f^n(C) \subset C_n \subset B_n, \qfor n\geq N,
\end{equation}
where
\begin{equation}
\label{Cndef}
C_n= A\left(r_n^{a_n+2\pi\delta_n}, r_n^{b_n(1-3\pi\delta_n)}\right), \text{ with } \delta_n= 1/\sqrt{\log r_n}.
\end{equation}
\end{lemma}
We also need the following, which shows that within $C_n$ the modulus of $f$ is very close to the maximum modulus, for large values of $n$. This is summarised from \cite[~Theorem~5.1(b)]{2011arXiv1109.1794B}.
\begin{lemma}
\label{L5.1}
Let $f, \ U$ and $z_0$ be as in Lemma~\ref{LThe1.5}. For large $n\isnatural$, let $r_n, \ a_n$ and $b_n$ be as in (\ref{Bneq}), and let $\delta_n=1/\sqrt{\log r_n}$. Then, there exists $N$ such that for $n\geq N$, and $m\isnatural$,
\begin{equation}
\label{L5.1eq}
\log |f^m(z)| \geq \left(1 - \delta_n\right) \log M(|z|,f^m), \qfor z \in A\left(r_n^{a_n+2\pi\delta_n}, r_n^{b_n-2\pi\delta_n}\right).
\end{equation}
\end{lemma}
The following is a straightforward consequence of these lemmas.
\begin{lemma}
\label{mconnlemma}
Let $f$ and $U$ be as in Lemma~\ref{LThe1.5}, let $z \in U$ and let $0 < c < 1$. Then there exists $N\isnatural$ such that
\begin{equation}
\label{eq1}
|f^{n+m}(z)| \geq M^m(|f^n(z)|^c), \qfor n\geq N, \ m\isnatural.
\end{equation}
\end{lemma}
\begin{proof}
Fix $z_0\in U$, and let $\beta_n= 1 - 1/\sqrt{\log r_n}$, where $r_n = |f^n(z_0)|$, for $n\isnatural$. It is well-known that it follows from Lemma~\ref{Lbaker} that $U \subset A(f)$. Hence there exists $\ell \isnatural$ such that $r_{n+\ell} \geq M^n(R)$, for $n \isnatural$, where $R > 0$ is such that $M^n(R)\rightarrow\infty$ as $n\rightarrow\infty$. It follows from (\ref{Meq1}) that we can choose $N\isnatural$ sufficiently large that 
\begin{equation}
\label{betaeq}
\prod_{k=N}^{\infty} \beta_{k} > c.
\end{equation}
Now let $z \in U$. We can further assume that $N$ is sufficiently large that $$|f^n(z)|^c> R_0, \qfor n~\geq~N,$$ where $R_0$ is the constant from Lemma~\ref{LMtheo}. Now, by Lemma~\ref{LThe1.5}, $$C_n \subset A\left(r_n^{a_n+2\pi\delta_n}, r_n^{b_n-2\pi\delta_n}\right),$$ for large values of $n$. Hence, we can assume, by Lemma~\ref{L1.3} and Lemma~\ref{L5.1}, that $N$ is sufficiently large that 
\begin{equation}
\label{intereq}
\log |f^{n+1}(z)| \geq \beta_n\log M(|f^n(z)|), \qfor n\geq N.
\end{equation}
Hence, by (\ref{intereq}) and (\ref{Msmalleq}),
\begin{equation}
\label{eq2}
|f^{n+1}(z)| \geq M(|f^n(z)|)^{\beta_n} \geq M(|f^n(z)|^{\beta_n}), \qfor n\geq N.
\end{equation}
By repeated application of (\ref{eq2}) and (\ref{Msmalleq}), and by (\ref{betaeq}), we have that
\begin{equation*}
|f^{n+m}(z)| \geq M^m(|f^n(z)|^{\prod_{k=0}^{m-1} \beta_{n+k}}) \geq M^m(|f^n(z)|^c), \qfor n \geq N, \ m\isnatural,
\end{equation*}
as required.
\end{proof}

%
%
We also need the following \cite[Theorem 2.3]{Rippon01102012}.
\begin{lemma}
\label{Lepoints}
Let $f$ be a {\tef} and let $\eta > 1$. There exists $R_0' = R_0'(f) > 0$ such that if $r > R_0'$, then there exists $$z'\in A(r, \eta r) \cap A(f)$$ with $$|f^n(z')| > M^n(r,f), \qfor n \isnatural,$$ and hence $$z' \in A_r(f) \ \ \text{and} \ \ M(\eta r, f^n) > M^n(r, f), \qfor n\isnatural.$$
\end{lemma}

%
%
Finally, in Section~\ref{SRA} we need the following well-known result \cite[Lemma 2.1]{pre05533139}.
\begin{lemma}
\label{Lblow}
Let $f$ be a transcendental entire function, let $K$ be a compact set with $K \cap E(f) = \emptyset$ and let $\Delta$ be an open {\nhd} of $z \in J( f)$. Then there exists $N\isnatural$ such that $f^n(\Delta) \supset K$, for $n \geq N$. 
\end{lemma}
Here $O^-(z) = \{w : f^n(w) = z, \text{ for some } n\isnatural\}$ and $E(f) = \{z : O^-(z) \text{ is finite}\}.$ The set $E(f)$ contains at most one point. 
%
%
%
%
\section{A map on a nested sequence of domains}
\label{Sblaschke}
In this section we prove a result about the existence and properties of a fixed point for certain {\tef}s. This may be of independent interest. For a hyperbolic domain $V$, we write $[w, z]_V$ for the hyperbolic distance between $w$ and $z$ in $V$. The main result of this section is as follows.
\begin{theo}
\label{TBlaschke}
Suppose that $f$ is a {\tef}, and that $(G_n)_{n\geq 0}$ is a sequence of bounded simply connected domains such that 
\begin{equation}
\label{Gneq}
\overline{G_n} \subset G_{n+1} \ \text{ and } \ f(\partial G_n) = \partial G_{n+1}, \qfor n=0, 1, 2, \cdots.
\end{equation}
Then there exists $\alpha\in G_0$, a fixed point of $f$, such that, if $K\subset G_0$ is compact, then $$[\alpha, f^n(z)]_{G_n} \rightarrow 0 \text{ as } n\rightarrow\infty, \quad\text{uniformly for } z \in K.$$
\end{theo}
%
%
To prove Theorem~\ref{TBlaschke} we require the following lemma. Define $\mathbb{D} = \{ z : |z| < 1\}$.
\begin{lemma}
\label{LBla}
Suppose that $(B_n)_{n\geq 0}$ is a sequence of analytic functions from $\mathbb{D}$ to $\mathbb{D}$. Suppose also that there exist $\alpha\in\mathbb{D}$  and $\lambda\in (0,1)$ such that
\begin{equation}
\label{Bndefeq}
B_n(\alpha) = \alpha \quad \text{ and } \quad |B_n'(\alpha)| \leq \lambda, \qfor n=0, 1, 2, \cdots.
\end{equation}
Then, if $K'$ is a compact subset of $\mathbb{D}$, $$B_n \circ \cdots \circ B_0(z) \rightarrow \alpha \text{ as } n\rightarrow\infty,\quad \text{uniformly for } z\in K'.$$
\end{lemma}
\begin{proof}
By conjugating with a M\"obius map if necessary, we may assume that $\alpha = 0$.

A result of Beardon and Carne \cite[p.217]{MR1216206} states that if $g: \mathbb{D} \to \mathbb{D}$ is an analytic function with $g(0) = 0$, then $$|g(z)| \leq |z|\left(\frac{|z| + |g'(0)|}{1+|g'(0)z|}\right), \qfor z \in \mathbb{D},$$ in which case $$M(r,g) \leq r\left(\frac{r + |g'(0)|}{1+|g'(0)|r}\right), \qfor r \in (0,1).$$ 

Now, $(r + x)/(1+rx)$ is an increasing function of $x$, for $r\in(0, 1)$. Hence, by (\ref{Bndefeq}),
\begin{equation}
\label{muthing}
M(r,B_n) \leq r\left(\frac{r + \lambda}{1+\lambda r}\right) = \mu(r) < r, \qfor r\in (0,1), \ n=0, 1, 2, \cdots.
\end{equation}

Note that $\mu(r)$ is a strictly increasing function of $r\in( 0,1)$. Let $r_0\in (0,1)$ be such that $|z| \leq r_0$, for $z \in K'.$ Then, by (\ref{muthing}), $$|B_{n-1} \circ \cdots \circ B_0 (z)| \leq \mu^n(r_0), \qfor z \in K'.$$ 

Now $\mu(0) = 0, \ \mu(1)=1$, and $0 < \mu(r) < r$, for $r\in(0,1)$. Hence $\mu^n(r_0)\rightarrow 0$ as $n\rightarrow\infty$. This completes the proof of the lemma.
\end{proof}
%
%
%
%
We now prove Theorem~\ref{TBlaschke}.
\begin{proof}[Proof of Theorem~\ref{TBlaschke}]
By (\ref{Gneq}), the triple $(f,G_0,G_{1})$ is a \itshape polynomial-like map \normalfont in the sense of Douady and Hubbard \cite{MR816367}. Hence, by \cite[Lemma 3]{MR1840089} (see also \cite[Lemma 3]{MR1771857}), there exists a point $\alpha~\in~G_0$ such that $f(\alpha) = \alpha$. 

For $n=0, 1, 2, \cdots$, let $\phi_n:\mathbb{D} \to G_n$ be a Riemann map such that $\phi_n(0) = \alpha$. Note that, since hyperbolic distance is preserved under conformal maps, we have that 
\begin{equation}
\label{mapseq}
[\alpha, f^n(z)]_{G_n} = [0, \phi_n^{-1} \circ f^n(z)]_{\mathbb{D}}, \qfor z \in G_0.
\end{equation}
Hence it suffices to show that $[0, \phi_n^{-1} \circ f^n(z)]_{\mathbb{D}}\rightarrow 0$ as $n\rightarrow\infty$, uniformly for $z \in K$.

Define functions $B_n: \mathbb{D} \to \mathbb{D}$ by $$B_n = \phi_{n+1}^{-1} \circ f \circ \phi_n, \qfor n=0, 1, 2, \cdots.$$ Then $B_n$ is a proper map such that $B_n(0) = 0$, for $n=0, 1, 2, \cdots$, and so is a finite Blaschke product 
\begin{equation}
B_n(z) = c_n z^{q_n} \prod_{k=1}^{p_n} \left(\frac{z - a_{k,n}}{1 - \overline{a_{k,n}}z}\right)^{m_{k,n}}, \qfor z \in \mathbb{D},
\end{equation}
where $p_n, q_n, m_{k,n} \isnatural$, $|c_n| = 1$, $0 < |a_{k,n}|< 1$, and $a_{k,n} = a_{k',n}$ implies that $k=k'$, for $n=0, 1, 2, \cdots,$ and $k=1,2,\cdots, p_n$. (See, for example, \cite[p.35]{MR0162913}.)

We claim that there exists $\lambda\in(0,1)$ such that
\begin{equation}
\label{lambdaeq}
|B_n'(0)| \leq \lambda, \qfor n=0, 1, 2, \cdots.
\end{equation}
Suppose first that $q_k \geq 2$ for some $k\isnatural$. Then $B_n'(0) = 0$, for $n=0, 1, 2, \cdots$.

Suppose, on the other hand, that $q_n = 1$, for $n=0, 1, 2, \cdots$. Then 
\begin{equation}
\label{Bneq1}
|B_n'(0)| = \prod_{k=1}^{p_n} |a_{k,n}|^{m_{k,n}} < 1, \qfor n=0, 1, 2, \cdots.
\end{equation}
For $n=0, 1, 2, \cdots,$ and $k=1,2,\cdots, p_n$, set $\alpha_{k,n} = \phi_n(a_{k,n})$. Then $f(\alpha_{k,n}) = \alpha$, and so $\phi_{n+1}^{-1}(\alpha_{k,n})$ is a zero of $B_{n+1}$. Without loss of generality we can assume that $$\alpha_{k,n} = \alpha_{k,n+1} = \phi_{n+1}(a_{k,n+1}), \qfor n=0, 1, 2, \cdots, \ k=1,2,\cdots, p_n.$$ Now, by (\ref{Gneq}), $[\alpha_{k,n},\alpha]_{G_n} \geq [\alpha_{k,n+1},\alpha]_{G_{n+1}}$. Hence, once again since hyperbolic distance is preserved under conformal maps, $|a_{k,n+1}| \leq |a_{k,n}|$. 

Moreover, $m_{k,n}$ and $m_{k,n+1}$ are both equal to the multiplicity of the zero $\alpha_{k,n}$ of $f(z) - \alpha$, and so $m_{k,n} = m_{k,n+1}$. Thus, by (\ref{Bneq1}),  $$|B_{n+1}'(0)| \leq |B_n'(0)|, \qfor n=0, 1, 2, \cdots.$$ This establishes (\ref{lambdaeq}).

Let $K\subset G_0$ be compact. By Lemma~\ref{LBla}, applied with $K' = \phi_0^{-1}(K)$, we obtain that $$[0, B_{n-1} \circ \cdots \circ B_1 \circ B_0 \circ \phi_0^{-1}(z)]_{\mathbb{D}}\rightarrow 0 \text{ as } n\rightarrow\infty,\quad \text{uniformly for } z\in K.$$ The result follows by (\ref{mapseq}), since $$B_{n-1} \circ \cdots \circ B_1 \circ B_0 \circ \phi_0^{-1} = \phi_n^{-1} \circ f^n.$$
\end{proof}
%
%
%
%
\section{The function $R_A$ defined in a multiply connected Fatou component}
\label{SRAinMC}
The main role of this section is to introduce the function $R_A$, which plays a key role in the proof of Theorem~\ref{Tloops}. Before stating and proving a sequence of lemmas, we outline how these results are used.

Suppose that $U$ is a multiply connected Fatou component which surrounds the origin. We show that if $U$ is sufficiently far from the origin, then we can define a real-valued function $R_A$ which, for each $z\in \overline{U}$, is the largest value of $R$ such that $z \in A_R(f)$; see (\ref{RAdef}) below. It turns out that this function has a close relationship to fundamental loops. Indeed, where defined, $R_A$ is strictly less than $R$ in $H_R$, and is at least equal to $R$ on $L_R$. We then prove that the function $R_A$ has certain continuity properties, and shares level sets with the function $h$ defined in (\ref{hdefinU}). These facts allow us to show that; 
\begin{enumerate}[(a)]
\item on $\intb U$, $R_A$ is equal to its infimum in $U$; 
\item on $\outb U$, $R_A$ is at least equal to its supremum in $U$;
\item $R_A$ does not achieve a maximum or a minimum in $U$.
\end{enumerate}
Because of the close relationship between the function $R_A$ and the definition of fundamental loops, properties (a), (b) and (c) above can then be used to prove Theorem~\ref{Tloops} parts (a), (b) and (c) respectively. Theorems~\ref{TloopsinJF} and \ref{Tloopsandh} then follow quickly.

We start with a simple lemma.
\begin{lemma}
\label{Lfixed}
Suppose that $f$ is a {\tef} and that $A_R(f)$ is a {\spw}. Then $f$ has a fixed point.
\end{lemma}
\begin{proof}
Suppose that $H_R$ is a fundamental hole of $f$. By \cite[Lemma 7.2]{Rippon01102012} we have that the triple $(f, H_R, f(H_R))$ is a polynomial-like map. The result follows by \cite[~Lemma 3]{MR1840089}.
\end{proof}
%
%
The following lemma is central to our results.
\begin{lemma}
\label{Lnhd}
Suppose that $f$ is {\tef}. Then there exists $R'=R'(f)>0$ such that the following holds. Suppose that $U$ is a multiply connected Fatou component of $f$, which surrounds the origin and satisfies \upshape\rmfamily~dist\itshape$(0,~U)~>~R'$. Define $G_n$ as the complementary component of $\overline{U_n}$ which contains the origin, for $n=0,1,2,\cdots$. Then
\begin{enumerate}[(a)]
\item $\overline{G_n} \subset G_{n+1}$,  for $n=0,1,2,\cdots$;
\item $f(\partial G_n) = \partial G_{n+1}$, for $n=0,1,2,\cdots$;
\item for all $z \in \overline{U}$ there exists $R = R(z)$ such that $z \in A_R(f)$.
\end{enumerate}
\end{lemma}
\begin{proof}
First we note \cite[Theorem 1.9 (a)]{Rippon01102012} that $A_R(f)$ is a {\spw}. Hence, by Lemma~\ref{Lfixed}, $f$ has a fixed point $\alpha$.

Let $V$ be a multiply connected Fatou component of $f$. By Lemma~\ref{Lbaker} there is an $N\isnatural$ such that $f^N(V)$ surrounds both the origin and $\alpha$, and also $f^{n+1}(V)$ surrounds $f^n(V)$ for $n\geq N$.

Choose $R>0$ such that $M^n(R)\rightarrow\infty$ as $n\rightarrow\infty$. Then \cite[Theorem~1.2(a)]{Rippon01102012} there is an $L\isnatural$ such that $\overline{f^L(V)} \subset A_R(f)$. Set $M = \max\{L, N\}$, and let $R'~=~\max~\{|z|~:~z\in\overline{f^M(V)} \}.$

Suppose that $U$ is any multiply connected Fatou component such that $U$ surrounds the origin and satisfies \rmfamily dist\normalfont$(0, U) > R'$. Then $U$ certainly surrounds $f^M(V)$. Since $U$ surrounds $\alpha$, then $U_1=f(U)$ surrounds $\alpha$ by the argument principle. Moreover, $U_1$ cannot meet either $f^{M+1}(V)$ or $U$, since $\partial U_1 \subset J(f)$. Hence, by the maximum principle, $U_1$ surrounds both $f^{M+1}(V)$ and $U$. Inductively, $U_{k}$ surrounds both $f^{M+k}(V)$ and $U_{k-1}$, for $k\isnatural$. Parts (a) and (c) of the lemma follow from this fact and the choice of $M$. 

Finally we establish part (b). Choose $n=0,1,2,\cdots$. Since $f(G_n)$ is open and connected, and its boundary is in $J(f)$, it cannot meet the boundary of $G_{n+1}$. Now, $\alpha\in G_n$, and so $f(G_n) \cap G_{n+1} \ne \emptyset$. Hence $\partial f(G_n)$ must lie in $\overline{G_{n+1}}$, and so $$f(\partial G_n)\subset\overline{G_{n+1}}.$$ Moreover, $f$ is a proper map on the Fatou component $U_n$, and so $$f(\partial G_n) \subset f(\partial U_n) = \partial U_{n+1}.$$ Thus $\partial f(G_n) = \partial G_{n+1},$ as required.
\end{proof}
%
%
Suppose that $U$ is a multiply connected Fatou component which surrounds the origin, and that \rmfamily dist\normalfont$(0, U)>R'$, where $R'$ is the constant from Lemma~\ref{Lnhd}. Then, by Lemma~\ref{Lnhd}(c) and by the continuity of $M$, we may define 
\begin{equation}
\label{RAdef}
R_A(z) =  \max \{R : z \in A_R(f)\}, \qfor z \in \overline{U}.
\end{equation}
The function $R_A$ has some strong continuity properties, and shares level sets with the function $h$.
%
%
\begin{lemma}
\label{Llimitedcontandlevel}
Suppose that $f$ and $U$ are as in Theorem~\ref{Tloops}, and that $R_A$ is as in (\ref{RAdef}). Then $R_A$ is {\usc} in $\overline{U}$ and continuous in $U$. Moreover, if $h$ is as in (\ref{hdefinU}), then there exists a continuous strictly increasing function $\phi:\mathbb{R} \to \mathbb{R}$ such that 
\begin{equation}
\label{phidef}
R_A(z) = \phi(h(z)), \qfor z \in U.
\end{equation}
\end{lemma}
\begin{proof}
We first prove that $R_A$ is {\usc} in $\overline{U}$. Suppose that $z\in\overline{U}$ and that $\epsilon > 0$. By the definition of $R_A$, we have that $z \notin A_{R_A(z)+\epsilon}(f)$. Hence there is an $N\isnatural$ such that $|f^N(z)| < M^N(R_A(z)+\epsilon)$. By continuity, there exists a $\delta>0$ such that $$|f^N(z')| < M^N(R_A(z)+\epsilon), \qfor z' \in B(z,\ \delta).$$ Hence $R_A(z') < R_A(z)+\epsilon$, for all $z'\in\overline{U}\cap B(z,\ \delta)$. This completes the proof that $R_A$ is {\usc} in $\overline{U}$. \\

To prove that $R_A$ is continuous in $U$ we need to prove that $R_A$ is {\lsc} at $z\in U$. Suppose, to the contrary, that $R_A$ is not {\lsc} at $z$. Then there exists $\epsilon>0$ such that the following holds. If $\Delta \subset U$ is a {\nhd} of $z$, then there is a $z' \in \Delta$ such that $R_A(z) - \epsilon > R_A(z')$, in which case $z' \notin A_{R_A(z)-\epsilon}(f)$. There exists, therefore, a sequence $(z_k)_{k\isnatural}$ of points of $U$, distinct from but tending to $z$, and a sequence $(n_k)_{k\isnatural}$ of integers such that $$|f^{n_k}(z_k)| < M^{n_k}(R_A(z)-\epsilon), \qfor k\isnatural.$$
Hence, for each $k\isnatural$,
\begin{equation*}
|f^{n_k}(z_k)| < M^{n_k}(R_A(z)-\epsilon) < M^{n_k}(R_A(z)) \leq |f^{n_k}(z)|, 
\end{equation*}
which implies that
\begin{equation}
\label{EQ1}
 \frac{\log M^{n_k}(R_A(z))}{\log M^{n_k}(R_A(z)-\epsilon)} < \frac{\log|f^{n_k}(z)|} {\log|f^{n_k}(z_k)|}, \qfor k\isnatural.
\end{equation}
We now establish a contradiction by showing that the right-hand side of (\ref{EQ1}) has an upper bound which tends to $1$ as $k\rightarrow \infty$, but the left-hand side is greater than some $c > 1$, for sufficiently large values of $k$. Note that we can assume that $n_k\rightarrow\infty$ as $k\rightarrow\infty$.

We may assume that $\epsilon$ is sufficiently small that $M^n(R_A(z)-\epsilon) \rightarrow\infty$ as $n\rightarrow\infty$. Hence we can choose $N$ large enough that $M^{N}(R_A(z)-\epsilon) > R_0$, where $R_0$ is the constant from Lemma~\ref{LMtheo}. Set $$r = M^{N}(R_A(z)-\epsilon) \ \text{ and } \ c = \frac{\log M^{N}(R_A(z))}{\log r} > 1.$$ It follows by repeated application of (\ref{Mbigeq}) that we have $$\log M^m(r^c) \geq c \log M^m(r), \qfor m\isnatural.$$ Hence
\begin{equation*}
\frac{\log M^{m+N}(R_A(z))}{\log M^{m+N}(R_A(z)-\epsilon)} = \frac{\log M^{m}(r^c)}{\log M^{m}(r)} \geq c > 1, \qfor m\isnatural.
\end{equation*}
This establishes our claim regarding the left-hand side of (\ref{EQ1}). \\

To establish our claim regarding the right-hand side of (\ref{EQ1}) we use some techniques from \cite{MR1767705}, though we give the full details for completeness. For a hyperbolic domain $V$, we let $\rho_{V}$ denote the density of the hyperbolic metric in $V$.

Choose any $w_1,w_2 \in J(f)$ with $w_1 \ne w_2$, and put $G = \mathbb{C} \backslash \{w_1,w_2\}$. Note that $$[z,z_k]_{U} \geq [f^{n_k}(z),f^{n_k}(z_k)]_{f^{n_k}(U)} \geq [f^{n_k}(z),f^{n_k}(z_k)]_{G} = \int_{\Gamma_k} \rho_G(z) \ |dz|, \qfor k\isnatural,$$ where $\Gamma_k$ is a hyperbolic geodesic in $G$ joining $f^{n_k}(z)$ to $f^{n_k}(z_k)$. By, for example, \cite[Theorem 9.14]{MR1049148}, there exist $R > 2$ and $C > 0$ such that $$\rho_G(z) \geq \frac{C}{|z|\log|z|}, \qfor |z| \geq R.$$

Choose $K$ sufficiently large such that $$|f^{n_k}(z)| > |f^{n_k}(z_k)| > 2R, \qfor k \geq K.$$ We then have
\begin{equation}
[z,z_k]_{U} \geq \int_{\Gamma_k} \rho_G(z) \ |dz| \geq C \int_{|f^{n_k}(z_k)|}^{|f^{n_k}(z)|} \frac{dr}{r \log r} = C\log\left(\frac{\log|f^{n_k}(z)|}{\log|f^{n_k}(z_k)|}\right).
\end{equation}
Hence 
\begin{equation}
\label{logeq}
\frac{\log|f^{n_k}(z)|}{\log|f^{n_k}(z_k)|} \leq \exp([z,z_k]_{U}/C), \qfor k\isnatural.
\end{equation}
As $k\rightarrow\infty$, $z_k\rightarrow z$ and so $[z,z_k]_{U}\rightarrow 0$. Hence the right-hand side of (\ref{logeq}) is indeed bounded above by a term tending to $1$ as $k\rightarrow\infty$. This completes the proof that $R_A$ is continuous in $U$. \\

Finally we need to prove that there exists a real function $\phi$ which satisfies (\ref{phidef}). Our method of proof is as follows. Suppose that $w, z \in U$. We claim that $h(w)~<~h(z)$ if and only if $R_A(w)~<~R_A(z)$. This, combined with the fact that both $h$ and $R_A$ are continuous in $U$, proves that $R_A(z) = \phi(h(z))$, for $z \in U$, where $\phi$ is continuous and strictly increasing.

Let $w, z \in U$. Suppose first that $R_A(w) < R_A(z) = r$, say. Then, there is an $N\isnatural$ such that $$|f^n(z)| \geq M^n(r) > |f^n(w)|, \qfor n\geq N.$$ Assume also that $N$ is sufficiently large that $|f^n(w)| > R_0$, for $n\geq N$, where $R_0$ is the constant in Lemma~\ref{LMtheo}. Set $$c = \frac{\log M^N(r)}{\log |f^N(w)|} >1.$$ Then, by (\ref{Mbigeq}),
\begin{align*}
M^{N+m}(r) = M^m(|f^N(w)|^c) \geq M^m(|f^N(w)|)^c \geq |f^{N+m}(w)|^{c}, \qfor m\isnatural.
\end{align*}
Hence
\begin{align*}
\frac{h(w)}{h(z)} &= \lim_{m\rightarrow\infty} \frac{\log |f^{N+m}(w)|}{\log |f^{N+m}(z)|} \leq \lim_{m\rightarrow\infty} \frac{\log |f^{N+m}(w)|}{\log |M^{N+m}(r)|} \leq \frac{1}{c} < 1,
\end{align*}
and so $h(w) < h(z)$. This completes the first part of the proof. \\

Suppose next that $h(w) < h(z)$. The proof is complete if we can show that $R_A(w)~<~R_A(z)$. Choose $c$ such that $h(w)/h(z) < c < 1$. Choose $N'$ sufficiently large such that $$|f^n(z)|^c~>~|f^n(w)|, \qfor n\geq N'.$$ By Lemma~\ref{mconnlemma}, there exists $N>N'$ such that
\begin{equation*}
|f^{N+m}(z)| \geq M^m(|f^{N}(z)|^{c}), \qfor m\isnatural.
\end{equation*}             
Hence
\begin{equation}
\label{e1}
R_A(f^N(z)) \geq |f^N(z)|^{c} > |f^N(w)| \geq M^N(R_A(w)).
\end{equation}
Set $R = M^{-N}(|f^N(z)|^{c})$, and note that $R>R_A(w)$ by (\ref{e1}). Then $$|f^{N+m}(z)| \geq M^m(|f^N(z)|^{c}) = M^{N+m}(R), \qfor m\isnatural.$$ Hence $R_A(z) \geq R > R_A(w)$ as required. This completes the proof of the lemma.
\end{proof}
%
%
\begin{remark}
\normalfont In fact, with the conditions of Lemma~\ref{Llimitedcontandlevel}, the stronger result holds that $R_A$ is continuous in $\overline{U}\backslash\outb U$. This follows from Lemma~\ref{LRAniceinmconn} below, but is not pertinent to the proofs of the results of this paper.
\end{remark}
We use Lemma~\ref{Llimitedcontandlevel} to prove the following result regarding the values of the function $R_A$ in $\overline{U}$.
\begin{lemma}
\label{LRAniceinmconn}
Suppose that $f$ and $U$ are as in Theorem~\ref{Tloops} and that $R_A$ is as in (\ref{RAdef}). Set 
\begin{equation}
\label{R12def}
R_1 = R_1(U) = \inf_{z \in U} R_A(z) \quad \text{ and } \quad R_2 = R_2(U) = \sup_{z\in U} R_A(z).
\end{equation}
Then
\begin{enumerate}[(a)]
\item $R_A(z) = R_1$, for $z \in\partial U\backslash\outb U$;
\item $R_A(z) \geq R_2$, for $z \in\outb U$;
\item $R_1 < R_A(z) < R_2$, for $z \in U$.
\end{enumerate}
\end{lemma}
\begin{proof}
First, suppose that $z \in\intb U$. Choose $w \in U$. By Lemma~\ref{L1.3} applied with $C = \{ w \}$, there exists $N\isnatural$ such that $|f^n(w)| > |f^n(z)|$, for $n\geq N$, in which case $R_A(z) \leq R_A(w)$. Hence $R_A(z) \leq R_1$. Equality follows by the upper semicontinuity of $R_A$ at $z\in\overline{U}$.

We now show, more generally, that $R_A(z) = R_1$ for $z \in \partial U\backslash\outb U$, by showing that $R_A$ is constant on this set. First, suppose that $$R_A(z) > R_0, \qfor z \in \partial U\backslash\outb U,$$ where $R_0$ is the constant from Lemma~\ref{LMtheo}. Suppose also that there exist points $z_1,z_2\in\partial U\backslash\outb U$ with $$R_A(z_1) = R > \rho > R_A(z_2), \qfor \text{some } \rho.$$ Set $c=\log R / \log \rho >1$. Then, for all sufficiently large $n\isnatural$, we have by (\ref{Mbigeq}),
\begin{equation}
\label{eqq}
|f^n(z_2)|^c < M^n(\rho)^c \leq M^n(\rho^c) = M^n(R) \leq |f^n(z_1)|.
\end{equation}
We now observe that, for sufficiently large values of $n$, 
\begin{equation}
\label{rneq}
r_n^{\underline{a}_n} \leq |f^n(z)| \leq r_n^{a_n}, \qfor z\in \partial U\backslash\outb U,
\end{equation}
where $\underline{a}_n$ is as in Lemma~\ref{L1.3}. This fact is in part of the proof of \cite[Theorem 1.6]{2011arXiv1109.1794B}, but we give a brief justification for completeness. Suppose that $K$ is a component of $\partial U\backslash\outb U$ and $\gamma$ is a Jordan curve in $U$ that contains $K$ in its interior \rom{int}$(\gamma)$. For large $n$ we have, by Lemma~\ref{L1.3}, that $f^n(\gamma) \subset C_n \subset B_n$. Hence $$f^n(\rom{\text{int}}(\gamma)) \subset \{z : |z| < r^{b_n} \},$$
and (\ref{rneq}) follows by the definitions of $\underline{a}_n$ and $a_n$, and the fact that $f^n(z) \notin B_n$, for $z\in \partial U\backslash\outb U$.

Now, by Lemma~\ref{LThe1.5}, both $\underline{a}_n$ and $a_n$ tend to $a$ as $n\rightarrow\infty$. Hence, for large values of $n\isnatural$, by (\ref{rneq}),
\begin{equation}
|f^n(z_2)|^c \geq r_n^{c\underline{a}_n} \geq r_n^{a_n} \geq |f^n(z_1)|,
\end{equation}
which is a contradiction to (\ref{eqq}).

Now suppose that $R_A(z) \leq R_0$, for some $z \in \partial U \backslash \outb U$. Let $U_n$, for some $n\isnatural$, be such that $R_A(z) > R_0$, for $z \in \partial U_n \backslash \outb U_n$. Now, $f^n$ is a proper map of $U$ to $U_n$ and $f^n(\partial U\backslash \outb U) = \partial U_n\backslash \outb U_n$. The result follows because $R_A(f^n(z)) = M^n(R_A(z))$ and since, by the above, $R_A$ is constant on $\partial U_n\backslash \outb U_n$. This completes the proof of part (a) of the lemma. \\

Next, suppose that $z \in\outb U$. Choose $w \in U$. By Lemma~\ref{L1.3} applied with $C = \{ w \}$, there exists $N\isnatural$ such that $|f^n(z)| > |f^n(w)|$, for $n\geq N$, in which case $R_A(z) \geq R_A(w)$. Thus $R_A(z) \geq R_2$ and this completes the proof of part (b) of the lemma. \\

Finally, suppose that there exists $z \in U$ such that $R_A(z) = R_2$, in which case $R_A$ achieves a maximum in $U$ at $z$. Then $h$ also achieves a maximum in $U$ at $z$, by Lemma~\ref{Llimitedcontandlevel}. This is a contradiction, because $h$ is harmonic in $U$. For a similar reason, $R_A$ cannot equal $R_1$ and so achieve a minimum in $U$. This completes the proof of the lemma.
\end{proof}
%
%
%
%
\section{Proofs of Theorem~\ref{Tloops}, Theorem~\ref{TloopsinJF} and Theorem~\ref{Tloopsandh}}
\label{Sloops}
In this section we prove Theorem~\ref{Tloops}, and then show that this can be used to prove Theorem~\ref{TloopsinJF} and Theorem~\ref{Tloopsandh}. We begin by proving the following result. Recall that $R_1 = R_1(U) = \inf_{z \in U} R_A(z)$.
\begin{lemma}
\label{LGnotin}
Suppose that $f$ and $U$ are as in Theorem~\ref{Tloops}, and let $G_0$ be the complementary component of $\overline{U}$ containing the origin. Then $G_0 \subset A_{R_1}(f)^c$.
\end{lemma}
\begin{proof}
Suppose, to the contrary, that there exists $z_0 \in G_0$ such that $z_0 \in A_{R_1}(f)$. Recall that $U_n=f^n(U)$, $G_n$ is the component of $\mathbb{C}\backslash\overline{U_n}$ containing the origin, and $\intb U_n = \partial G_n$. Let $r_n~=~\text{\rmfamily dist\normalfont}(0, \intb U_n)$ and let $z_n = f^n(z_0) \in G_n$, for $n\isnatural$.

In view of Lemma~\ref{Lnhd}(a) and (b), we can apply Theorem~\ref{TBlaschke}, with $G_n$ as above and with $K = \{z_0\}$. We obtain that $f$ has a fixed point $\alpha\in G_0$ such that $$[\alpha, z_n]_{{G_n}} \rightarrow 0 \text{ as } n\rightarrow\infty.$$

We claim that there exists $N\isnatural$ such that $|z_n| < r_n/2$ for $n\geq N$. Suppose, to the contrary, that $|z_n| > r_n/2$ infinitely often. For these values of $n$, let $\gamma_n$ be a curve in $G_n$ joining $\alpha$ and $z_n$ such that $$2[\alpha, z_n]_{{G_n}} \geq \int_{\gamma_n} \rho_{G_n}(w) |dw|.$$ Recall (for example, \cite[Theorem 4.3]{MR1230383}) that $$\rho_{G_n}(w) \geq \frac{1}{2 \ \text{\rmfamily dist\normalfont}(w, \partial G_n)},\qfor w \in G_n.$$

We can assume that $n$ is sufficiently large that $|\alpha| < r_n/4$. Let $$\gamma_n' = \gamma_n \cap B(0, r_n/2).$$ Note that \rmfamily dist\normalfont$(w, \partial G_n) \leq 2r_n$, for $w \in \gamma_n'$. Moreover the length of $\gamma_n'$ is certainly at least equal to $r_n/4$. Hence $$2[\alpha, z_n]_{{G_n}} \geq \int_{\gamma_n'} \rho_{G_n}(w) |dw| \geq  \frac{1}{4r_n} \int_{\gamma_n'} |dw| \geq \frac{1}{4r_n} \frac{r_n}{4} = \frac{1}{16},$$ which is a contradiction. Thus our claim is established. \\

We now set $\eta = 3/2$. By Lemma~\ref{Lepoints} and the above, there exists $N\isnatural$ such that the following conditions both hold. Firstly, there exists $z' \in A({r_N}/2, 3{r_N}/4)$ such that $z' \in A_{r_N/2}(f)$. Secondly, $|z_N| < r_N/2$. 

We have supposed that $z_0 \in A_{R_1}(f)$, and so this second condition implies that $M^N(R_1)~<~r_N/2$. Suppose that there exists $w \in \intb U_N \cap A_{r_N/2}(f)$. Since $w = f^N(w')$, for some $w' \in \intb U$, then $$w' \in A_{M^{-N}(r_N/2)}(f).$$
This is impossible since $M^{-N}(r_N/2) > R_1$, but $R_A(w') = R_1$, by Lemma~\ref{LRAniceinmconn}(a). Hence we have that $\intb U_N \cap A_{r_N/2}(f) = \emptyset$. This is a contradiction because $\intb U_N$ surrounds $z'$, but $A_{r_N/2}(f)$ has no bounded components \cite[Theorem 1.1]{Rippon01102012}. 
\end{proof}
%
%
We now prove Theorem~\ref{Tloops}.
\begin{proof}[Proof of Theorem~\ref{Tloops}]
First we let $R'$ be the constant from Lemma \ref{Lnhd}. Suppose that $U$ is a multiply connected Fatou component of $f$, such that $U$ surrounds the origin and \rmfamily dist\normalfont$(0, U) \geq R'$. Let $R_1=R_1(U)$ and $R_2=R_2(U)$ be the constants from (\ref{R12def}). Part (a) of the theorem, that $\intb U$ is the fundamental loop $L_{R_1}$, follows because $\intb U \subset A_{R_1}(f)$, by Lemma~\ref{LRAniceinmconn}(a), but the bounded component of $\mathbb{C}\backslash\intb U$ is in $A_{R_1}(f)^c$, by Lemma~\ref{LGnotin}. \\

Part (b) of the theorem, that $\outb U$ is the fundamental loop $L_{R_2}$, follows immediately from Lemma~\ref{LRAniceinmconn}(b) and (c). \\

Finally we prove part (c) of the theorem. Suppose that $L_R$ is a fundamental loop, and that $z \in L_R \cap U$. Now $z \in A_R(f)$, and so $R_A(z) \geq R$. Moreover, $R_A(w) < R$, for $w \in H_R \cap U$. Hence, by the continuity of $R_A$ in $U$, $R_A(z) = R$. Thus, by Lemma~\ref{LRAniceinmconn}(c), $R_1 < R_A(z) = R < R_2$. 

Recall that $R_A(w) = R_1$, for $w \in \partial U\backslash\outb U$. It follows, by the upper semicontinuity of $R_A$ in $\overline{U}$, that $L_R\cap\partial U\backslash\outb U = \emptyset$. 

It remains to show that $L_R \cap \outb U = \emptyset$. Suppose, to the contrary, that $L_R$ intersects $\outb U$. We note \cite[Lemma 7.2 (c)]{Rippon01102012} that, in general, if $L_\rho$ is a fundamental loop then $f(L_\rho) = L_{M(\rho)}$. By Lemma~\ref{L1.3}, applied to any closed subset of $L_R \cap U$, there exists $N\isnatural$ such that $L_{M^n(R)} \cap C_n \ne \emptyset$, where $C_n$ is the annulus defined in (\ref{Cndef}), for $n\geq N$. 

Next choose $\eta > 1$. We can assume that $N$ is sufficiently large that, for $n \geq N$ and $z \in \overline{U_n}$, we have that $|z| > \max\{R_0, R_0'\}$, where $R_0$ is the constant from Lemma~\ref{LMtheo} and $R_0'$ is the constant from Lemma~\ref{Lepoints}. We can also assume that $N$ is sufficiently large that the conclusions of Lemma~\ref{L5.1} can be applied.

Define $c_n = b_n - 2\pi\delta_n - \delta_n^2,$ for $n\isnatural.$ We can further assume that $N$ is sufficiently large that we have both $$b_N(1-3\pi\delta_N) < c_N < b_N - 2\pi\delta_N$$ and $$\eta r_{N}^{ b_{N}(1-3\pi\delta_{N})} < r_{N}^{c_N(1-\delta_N)}.$$

The first inequality is easy to satisfy since, by Lemma~\ref{LThe1.5}, $b_n\rightarrow b > 1$, as $n\rightarrow\infty$. The second can be satisfied since $$\eta r_{n}^{ b_{n}(1-3\pi\delta_{n})} = r_{n}^{b_{n}(1-3\pi\delta_{n}) + \log \eta \ \delta_n^2},$$ and $$c_n(1-\delta_n) = b_n(1 - (2\pi/b_n+1)\delta_n) + \delta_n^2(2\pi-1+\delta_n),$$ and since $\delta_n\rightarrow 0$ and $b_n\rightarrow b > 1$ as $n\rightarrow\infty$. \\

Consider the fundamental loop $L_{{M^{N}}(R)}$. Since $L_{M^{N}(R)} \cap C_{N} \ne \emptyset$, there is a point on $L_{M^{N}(R)}$ of modulus less than $r_{N}^{b_{N}(1-3\pi\delta_{N})}$. Hence 
\begin{equation}
\label{aneq}
M^{N}(R) < r_{N}^{b_n(1-3\pi\delta_{N})}.
\end{equation}

Moreover, by assumption we have that $L_{M^{N}}(R) \cap \outb U_N  \ne \emptyset$. Hence $L_{M^{N}(R)}$ surrounds points in $U_{N}$ which lie at all radii in $(r_{N}^{b_{N}(1-3\pi\delta_{N})}, r_{N}^{b_{N}})$. In particular, there exists a point 
\begin{equation}
\label{zisinHn}
z \in H_{M^{N}(R)} \cap U_N, \quad \text{ such that } |z| = r_{N}^{c_N}.
\end{equation}
Then, by Lemma~\ref{L5.1}, Lemma~\ref{LMtheo} and Lemma~\ref{Lepoints}, we have that, for $m\isnatural$,
\begin{align*}
|f^m(z)| &\geq M(r_N^{c_N}, f^m)^{1-\delta_N} \\ 
         &\geq M(r_N^{c_N(1-\delta_N)},f^m) \geq  M(\eta r_{N}^{ b_{N}(1-3\pi\delta_{N})} ,f^m) \\
         &\geq M^m(r_{N}^{b_{N}(1-3\pi\delta_{N})},f).
\end{align*}         
Hence $z \in A_\rho(f)$, where $\rho = r_{N}^{b_{N}(1-3\pi\delta_{N})}$. This is in contradiction to (\ref{aneq}), since $z\notin A_{M^N(R)}(f)$ by (\ref{zisinHn}). This completes the first half of the proof of part (c). \\

Finally, suppose that $R_1 < R < R_2$. Then, by the continuity of $R_A$ and the definitions of $R_1$ and $R_2$, there exists $z \in U$ such that $R_A(z) = R$. Hence the fundamental loop $L_R$ must intersect $U$, and so $L_R \subset U$. This completes the proof.
\end{proof}
%
%
Next we prove Theorem~\ref{TloopsinJF}, which states that if $f$ is a {\tef} and that $L_R$ is a fundamental loop of $f$, then either $L_R \subset F(f)$ or $L_R~\subset~J(f)$.
\begin{proof}[Proof of Theorem~\ref{TloopsinJF}]
Suppose first that $z\in L_R \ \cap \ U$, where $U$ is a simply connected Fatou component of $f$. Since $L_R~\subset A_R(f)$, it follows from \cite[Theorem~1.2(b)]{Rippon01102012} that $\overline{U} \subset A_R(f)$. This is a contradiction since $L_R = \partial H_R$ and $H_R \subset A_R(f)^c$. Hence $L_R$ cannot intersect any simply connected Fatou component of $f$.

Next suppose that $z \in L_R \ \cap \ U$, where $U$ is a multiply connected Fatou component of $f$. Then there exists $N\isnatural$ such that \rmfamily dist\normalfont$(0, U_N) > R'$, where $R'$ is the constant from Theorem~\ref{Tloops} and $U_N=f^N(U)$. Then $f^N(L_R)=L_{M^N(R)}$ is a fundamental loop which intersects $U_N$ and so, by Theorem~\ref{Tloops}, is contained in $U_N$. The result follows.
\end{proof}
%
%
Finally we prove Theorem~\ref{Tloopsandh}, which relates fundamental loops lying in $U$ to level sets of $h$.
\begin{proof}[Proof of Theorem~\ref{Tloopsandh}]
First suppose that $L_R\subset U$ is a fundamental loop. Then, because of the continuity of $R_A$ in $U$, we have $R_A(z) = R$, for $z\in L_R$. Hence, by Lemma~\ref{Llimitedcontandlevel}, $h$ is also constant on $L_R$. This completes the first part of the proof.

Suppose next that $\Gamma$ is a level set of $h$. By Lemma~\ref{Llimitedcontandlevel}, $\Gamma$ is also a level set of $R_A$, and so $R_A(z) =R$, say, for $z \in \Gamma$. Now $R_1 < R < R_2$, where $R_1$ and $R_2$ are as in (\ref{R12def}), and so, by Theorem~\ref{Tloops}, there is a fundamental loop $L_R \subset U$. The result follows, since $L_R \subset \Gamma$, again by Lemma~\ref{Llimitedcontandlevel}.
\end{proof}
%
%
%
%
\section{The function $R_A$ defined in $\mathbb{C}$}
\label{SRA}
The function $R_A$ played a key role in proving Theorem~\ref{Tloops}. In general, however, $R_A(z)$ cannot be defined for many values of $z\in A(f)$; consider, for example, $f(z) = e^z$ and $z = \log 2\pi + i\pi/2$. In this section we show that, with a certain normalisation of $f$, the definition of $R_A(z)$ can be extended in a natural way to all $z\iscomplex$. The function $R_A$ then has several interesting properties.

First we adopt the normalisation $f(0) = 0$. We observe that, by Lemma~\ref{Lfixed}, all {\tef}s for which $A_R(f)$ is a {\spw} have a fixed point and so, in this case, this normalisation is merely a change of coordinates. This suggests that this normalisation is not entirely unnatural when $A_R(f)$ is a {\spw}. Even when $f$ does not have a fixed point the normalisation $f(0) = 0$ is not as limiting as it might seem. If $f(0) \ne 0$, then we choose $\alpha$, a fixed point of $f^2$, and replace $f$ by $g$ where $g(z) = f^2(z+\alpha)-\alpha$. Then $g(0) = 0$, and the sets $A(f)$ and $A(g)$ differ only by a translation, since \cite[Theorem 2.6]{Rippon01102012} we have $A(f^2) = A(f)$.

With the normalisation $f(0) = 0$, we can define 
\begin{equation}
\label{Rfdef}
R_f = \max \{ r\geq 0 : M^n(r) \nrightarrow \infty \text{ as } n \rightarrow \infty\}.
\end{equation}
The following gives an alternative characterisation of $A(f)$ as a continuous limit of the closed sets $A_R(f)$.
\begin{theo}
\label{Tnewdef}
Suppose that $f$ is a {\tef}, that $f(0) = 0$, and that $R_f$ is as defined in (\ref{Rfdef}). Then
\begin{equation}
\label{AwithAR}
A(f) = \bigcup_{R>R_f} A_R(f).
\end{equation}
\end{theo}
\begin{proof}
If $z \in \bigcup_{R>R_f} A_R(f)$, then $z \in A_R(f)$ for some $R$ such that $M^n(R)\rightarrow\infty$ as $n\rightarrow\infty$, and so $z \in A(f)$ by definition.

Now, suppose that $z \in A(f)$. Then, by (\ref{Adef}), $f^\ell(z) \in A_R(f)$, for some $R>R_f$ and some $\ell\isnatural$. Note next that, since $f(0) = 0$, we have that $M^{-n}(r)$ is defined for all $r\geq 0$ and $n\isnatural$. Hence we can set $R' = M^{-\ell}(R)$, and we note that $R' > R_f$. Then $z \in A_{R'}(f)$ and so $z \in \bigcup_{R>R_f} A_R(f)$, as required. 
\end{proof}
For a {\tef} $f$ with $f(0) = 0$, we extend the definition of $R_A$ to the whole complex plane by setting
\begin{equation}
\label{fullRA}
R_A(z) =
  \begin{cases}
   \max \{R : z \in A_R(f)\}, &\qfor z \in A(f), \\
   R_f, &\qfor z \notin A(f).
  \end{cases}
\end{equation}
The existence of the maximum, for $z\in A(f)$, follows from (\ref{AwithAR}) and the continuity of $M$. Note that it follows from (\ref{fullRA}) that
\begin{equation}
\label{MinvertsRA}
M(R_A(z)) = R_A(f(z)), \qfor z\in A(f). 
\end{equation}
If $f$ satisfies the normalization $f(0) = 0$, then a stronger version of Lemma~\ref{Llimitedcontandlevel} holds.
\begin{theo}
\label{TRAcont}
Suppose that $f$ is a transcendental entire function and that $f(0)~=~0$. Then
\begin{enumerate}[(a)]
\item $R_A$ is {\usc} in $\mathbb{C}$;
\item $R_A$ is nowhere continuous in $A(f) \cap J(f)$;
\item $R_A$ is constant in a simply connected Fatou component of $f$;
\item $R_A$ is continuous in $A(f)^c \cup F(f)$.
\end{enumerate}
\end{theo}
\begin{proof}
Part (a) follows in exactly the same way as the first part of the proof of Lemma~\ref{Llimitedcontandlevel}, and so we omit the details.

Now we prove part (b). Observe that, in general, if $w\in A(f)$, $n>1$ and $f^n(w') = w$, then 
\begin{equation}
\label{RAineq}
R_A(w') = M^{-n}(R_A(w)) < M^{-1}(R_A(w)) < R_A(w).
\end{equation}

Suppose that $z\in A(f)\cap J(f)$ and assume first that $z\notin E(f)$. Let $\Delta$ be a {\nhd} of $z$, sufficiently small that $\overline{\Delta}\cap E(f) = \emptyset$. Then, by Lemma~\ref{Lblow}, there is an $n>1$ such that $f^n(\Delta) \supset \overline{\Delta}$. Hence, since $z$ cannot be periodic, there is a $z' \in \Delta$ with $z' \ne z$ and such that $f^n(z') = z$. Hence, by (\ref{RAineq}), 
\begin{equation}
\label{ineqforz}
R_A(z') < M^{-1}(R_A(z)) < R_A(z).
\end{equation}
This shows that $R_A$ is not continuous at $z$ in the case that $z \notin E(f)$, since $\Delta$ was arbitrary.

In the case that $z\in E(f)$, we first observe that $f(z)\notin E(f)$. Let $\Delta$ be a {\nhd} of $z$, sufficiently small that $\overline{f(\Delta)}\cap E(f) = \emptyset$. By the same argument as above, there is a $z' \in \Delta$ such that 
\begin{equation}
\label{ineqforz2}
R_A(f(z')) < M^{-1}(R_A(f(z))) < R_A(f(z)).
\end{equation}
Equation (\ref{ineqforz}) now follows from (\ref{ineqforz2}) and (\ref{MinvertsRA}). This completes the proof of part~(b).

Next we prove part (c). Suppose that $U$ is a simply connected Fatou component and that $U \cap A(f) = \emptyset$. Then $R_A(z) = R_f$, for $z \in U$. On the other hand, suppose that $z~\in~U~\cap~A_R(f)$, for some $R>R_f$. Then $\overline{U} \subset A_R(f)$ \cite[Theorem 1.2(b)]{Rippon01102012}. This completes the proof of part (c).

Finally we prove part (d). The result when $z \in A(f)^c$ is immediate from part~(a), and the fact that $R_A$ achieves its global minimum of $R_f$ everywhere in $A(f)^c$. If $z~\in~A(f)~\cap~F(f)$, then we can assume that $z$ is in a multiply connected Fatou component of $f$, and the proof follows in exactly the same way as the second part of the proof of Lemma~\ref{Llimitedcontandlevel}.
\end{proof}
In a multiply connected Fatou component, we can say more about the properties of the function $R_A$.
\begin{theo}
\label{TRAsubh}
Suppose that $f$ is a transcendental entire function and that $f(0)~=~0$. Then the function $v = -\log R_A$ is subharmonic in $F(f)$.
\end{theo}
\begin{remarkk}
\normalfont 
It follows from Theorem~\ref{TRAcont}(d) and Theorem~\ref{TRAsubh} that $1/R_A$ is in the class PL in each component of $F(f)$. Here (see \cite{MR1501707}), a function $u$ in a domain $D$ is said to be in the class PL if $u$ is continuous and non-negative in $D$, and $\log u$ is subharmonic in the part of $D$ where $u > 0$. This class is a generalisation of functions of the form $|\phi|$, where $\phi$ is analytic in $D$. The weaker result that $1/R_A$ is subharmonic in $F(f)$ also follows from Theorem~\ref{TRAsubh}, since $1/R_A(z) = \exp(v(z))$ and by \cite[Corollary 2.6.4]{MR1334766}.
\end{remarkk}
\begin{remarkk}
\normalfont It seems natural to ask if $v$ is harmonic in $F(f)$. This cannot be the case in general. For, by the last statement of Lemma~\ref{Llimitedcontandlevel}, if $v$ is harmonic in a multiply connected Fatou component $U$ which satisfies the conditions of Theorem~\ref{Tloops}, then there is a continuous function $\psi: \mathbb{R}\to\mathbb{R}$ such that
\begin{equation}
\label{aveq}
v(z) = \psi(h(z)), \qfor z \in U.
\end{equation}
If $v$ is harmonic, then we can differentiate (\ref{aveq}) to obtain that $\psi''(h(z)) = 0,$ for $z \in U$. Hence $v$ is a linear function of $h$ in $U$. Now, $v$ is finite in $\overline{U}$. In \cite[Example 2 and Theorem 1.6]{2011arXiv1109.1794B} it is shown that there exist {\tef}s such that $h$ is unbounded in $U$. For these functions the relationship between $h$ and $v$ cannot, therefore, be linear, and so $v$ is not harmonic in $U$.
\end{remarkk}
In order to prove Theorem~\ref{TRAsubh} we need three further lemmas. The first concerns repeated iteration of the function $M^{-1}$.
\begin{lemma}
\label{lemwascor1}
Suppose that $f$ is a {\tef} and that $f(0) = 0$. For each $n\isnatural$, define the function $v_n$ by
\begin{equation}
v_n(z) = -\log M^{-n}(|f^n(z)|), \qfor z \in D_n = \{z : f^n(z) \ne 0 \}.
\end{equation}
Then $v_n$ is subharmonic in $D_n$.
\end{lemma}
\begin{proof}
Since $\psi(s) = \log M^{-1}(e^s)$ is a concave and increasing function of $s$, we have (see, for example, \cite[Theorem 7.2.1]{MR2467621}) that $$\psi^n(s) = \log M^{-n}(e^s)$$ is also a concave function of $s$, for $n\isnatural$. Now, for each $n\isnatural$, $\log |f^n(z)|$ is a harmonic function of $z$ in $D_n$, since $f^n(z) \ne 0$ in $D_n$. The result follows since $$v_n(z) = -\log M^{-n}(\exp(\log |f^n(z)|)) = -\psi^n(\log |f^n(z)|),$$ is a convex function of a harmonic function; see e.g. \cite[p.47]{MR1334766}. 
\end{proof}
Note that, if $f(0) = 0$, then $0 \notin A(f)$ and so $v_n(z)$ is defined for all $z \in A(f)$ and $n\isnatural$.

The second lemma gives an alternative characterisation of the function $R_A$ in $A(f)$.
\begin{lemma}
\label{lem1}
Suppose that $f$ is a {\tef} and that $f(0) = 0$. Then, for each $z \in A(f)$, $(M^{-n}(|f^n(z)|))_{n\isnatural}$ is a non-increasing sequence, with limit $R_A(z)$.
\end{lemma}
\begin{proof}
Suppose that $z \in A(f)$. Since $M(|f^n(z)|) \geq |f^{n+1}(z)|$ we have that 
\begin{equation}
\label{mudecr}
M^{-n}(|f^n(z)|) \geq M^{-(n+1)}(|f^{n+1}(z)|), \qfor n\isnatural.
\end{equation}
Hence the sequence $(M^{-n}(|f^n(z)|))_{n\isnatural}$ is non-increasing. In addition, since \\
$|f^n(z)|~\geq~M^n(R_A(z))$, for $n\isnatural$, we have that $$M^{-n}(|f^n(z)|) \geq R_A(z), \qfor n\isnatural.$$ So $\lim_{n\rightarrow\infty} M^{-n}(|f^n(z)|)$ exists and is at least $R_A(z)$. It is straightforward to show from (\ref{mudecr}) that if this limit is $R$, then $z \in A_R(f)$. This completes the proof.
\end{proof}
We also need a result on subharmonic functions. Suppose that $D$ is a domain, and $u: D \to [-\infty, \infty)$ is a function which is locally bounded above in $D$. The \itshape {\usc} regularization \normalfont of $u$, $u^*: D \to [-\infty, \infty)$ is defined by $$u^*(z) = \limsup_{w\rightarrow z} u(w).$$ It can be shown that $u^*$ is the least {\usc}  function on $D$ such that $u^* \geq u$. The result we require is the following \cite[Theorem 3.4.2(a)]{MR1334766}.
\begin{lemma}[Brelot-Cartan Theorem]
\label{Lbc}
Suppose that $D$ is a domain, that $\mathcal{V}$ is a family of subharmonic functions on $D$ and that $u = \sup_{v\in\mathcal{V}} v$ is locally bounded above on $D$. Then $u^*$ is subharmonic on $D$.
\end{lemma}
We now give the proof of Theorem~\ref{TRAsubh}, that the function $v(z) = -\log R_A(z)$ is subharmonic, for $z \in F(f)$. 
\begin{proof}[Proof of Theorem~\ref{TRAsubh}]
Suppose that $z \in A(f)^c \cap F(f)$. The result follows because $R_A$ is constant in a {\nhd} of $z$. On the other hand, suppose that we have $z \in A(f) \cap F(f)$, and let $U$ be the Fatou component containing $z$. Since $R_A$ is constant in any simply connected Fatou component, we can assume that $U$ is multiply connected. Observe that, by Lemma~\ref{LRAniceinmconn}, applied, if necessary, to $U_N$ for some large $N$, there exists $R_1 > 0$ such that $R_A(z) \geq R_1$, for $z\in\overline{U}$. Hence $v$ is bounded above in $\overline{U}$. 

Let $v_n$ be as defined in Lemma~\ref{lemwascor1}. Then, by Lemma~\ref{lem1} and Lemma~\ref{lemwascor1}, $v_n$ is a non-decreasing sequence of subharmonic functions, converging pointwise in $\overline{U}$ to $v$. Hence, $\sup_{n\isnatural} v_n = v$. By Lemma~\ref{Lbc}, applied with $\mathcal{V} = \{v_n\ : n\isnatural\}$, $v^*$ is subharmonic in $U$. By Theorem~\ref{TRAcont} part (d), $v$ is continuous in $U$, and so $v^* = v$ there. This completes the proof.
\end{proof}
Another advantage of the normalisation $f(0) = 0$ is that, if this condition is satisfied, then the conclusions of Theorems~\ref{Tloops} and~\ref{Tloopsandh} hold for any multiply connected Fatou component which surrounds the origin, without the additional restriction of being a sufficient distance from the origin. This fact follows from the proof of Theorem~\ref{Tloops} and from the following version of Lemma~\ref{Lnhd}. \begin{lemma}
\label{Lnotnhd}
Suppose that $f$ is {\tef} and that $f(0) = 0$. Suppose that $U$ is a multiply connected Fatou component of $f$ which surrounds the origin, and define $G_n$ as the complementary component of $\overline{U_n}$ which contains the origin, for $n=0,1,2,\cdots$. Then
\begin{enumerate}[(a)]
\item $\overline{G_n} \subset G_{n+1}$,  for $n=0,1,2,\cdots$;
\item $f(\partial G_n) = \partial G_{n+1}$, for $n=0,1,2,\cdots$;
\item for all $z \in \overline{U}$ there exists $R = R(z)$ such that $z \in A_R(f)$.
\end{enumerate}
\end{lemma}
\begin{proof}
Parts (a) and (b) follow as in the proof of Lemma~\ref{Lnhd}, since the origin is a fixed point of $f$. Part (c) follows from Theorem~\ref{Tnewdef}.
\end{proof}
%
%
Given a {\tef} $f$ and $N\isnatural$, it is not hard to show that there is a point $z\in A(f)$ such that $|f^{n+1}(z)|$ is small compared to $|f^n(z)|$, for $n\leq N$. Hence $R_A(z)$ can be much smaller than $|z|$. It does seem reasonable, however, to expect that $M^n(R_A(z))$ should be comparable to $|f^n(z)|$, for large values of $n\isnatural$. We use results from \cite{2011arXiv1109.1794B} to prove the following.
\begin{theo}
\label{limittheo}
Suppose that $f$ is a {\tef}, that $f(0) = 0$, and that $z$ is in a multiply connected Fatou component of $f$. Then
\begin{equation}
\lim_{n\to\infty} \frac{\log |f^n(z)|}{\log M^n(R_A(z))} = 1.
\end{equation}
\end{theo}
\begin{proof}
Let $U$ be the Fatou component containing $z$. It follow from (\ref{MinvertsRA}) that we need to prove that 
\begin{equation}
\label{toprove}
\lim_{n\to\infty} \frac{\log |f^n(z)|}{\log R_A(f^n(z))} = 1.
\end{equation} 
By definition $|f^n(z)| \geq R_A(f^n(z))$, for $n\isnatural$. Suppose that, contrary to (\ref{toprove}), there exists $0 < c < 1$ and a sequence of natural numbers $(n_k)_{k\isnatural}$ such that $|f^{n_k}(z)|^c > R_A(f^{n_k}(z))$, for $k\isnatural$, and $n_k\rightarrow\infty$ as $k\rightarrow\infty$. Then, by the definition of $R_A$, for each $k\isnatural$ there exists $m_k\isnatural$ such that
\begin{equation}
\label{mkeq}
|f^{n_k+m}(z)| < M^m(|f^{n_k}(z)|^c), \qfor m\geq m_k.
\end{equation}
Since $n_k\rightarrow\infty$ as $k\rightarrow\infty$, we see that (\ref{mkeq}) is contrary to Lemma~\ref{mconnlemma}. This completes the proof of Theorem~\ref{limittheo}. 
\end{proof}
\setcounter{remarkk}{0}
\begin{remarkk}
\normalfont If $f$ is a {\tef} and $f(0) \ne 0$, then the conclusions of Theorem~\ref{TRAsubh} and Theorem~\ref{limittheo} still hold for a multiply connected Fatou component $U$ which satisfies the conditions of Theorem~\ref{Tloops}. This is readily seen from a review of the proofs of these results.
\end{remarkk}
\begin{remarkk}
\normalfont The only known examples of simply connected fast escaping Fatou components are given in \cite{MRunknown} and \cite{MR2959920}. It can be shown that, for the example in \cite{MR2959920}, if $z$ is in one of the simply connected fast escaping Fatou components then we have the stronger result that $$\lim_{n\to\infty} \frac{|f^n(z)|}{M^n(R_A(z))} = 1.$$ It would be interesting to know, in general, whether a result similar to Theorem~\ref{limittheo} holds for simply connected fast escaping Fatou components.
\end{remarkk}
\itshape Acknowledgment: \normalfont
The author is grateful to Gwyneth Stallard and Phil Rippon for all their help with this paper.
%
%
\bibliographystyle{acm}
\bibliography{Main}

\begin{thebibliography}{10}

\bibitem{MR759304}
{\sc Baker, I.~N.}
\newblock Wandering domains in the iteration of entire functions.
\newblock {\em Proc. London Math. Soc. (3) 49}, 3 (1984), 563--576.

\bibitem{MR1840089}
{\sc Bargmann, D., and Bergweiler, W.}
\newblock Periodic points and normal families.
\newblock {\em Proc. Amer. Math. Soc. 129}, 10 (2001), 2881--2888.

\bibitem{MR1216206}
{\sc Beardon, A.~F., and Carne, T.~K.}
\newblock A strengthening of the {S}chwarz-{P}ick inequality.
\newblock {\em Amer. Math. Monthly 99}, 3 (1992), 216--217.

\bibitem{MR1501707}
{\sc Beckenbach, E.~F., and Rad{\'o}, T.}
\newblock Subharmonic functions and minimal surfaces.
\newblock {\em Trans. Amer. Math. Soc. 35}, 3 (1933), 648--661.

\bibitem{MR1216719}
{\sc Bergweiler, W.}
\newblock Iteration of meromorphic functions.
\newblock {\em Bull. Amer. Math. Soc. (N.S.) 29}, 2 (1993), 151--188.

\bibitem{MRunknown}
{\sc Bergweiler, W.}
\newblock An entire function with simply and multiply connected wandering
  domains.
\newblock {\em Pure Appl. Math. Quarterly 7}, 2 (2011), 107--120.

\bibitem{MR1684251}
{\sc Bergweiler, W., and Hinkkanen, A.}
\newblock On semiconjugation of entire functions.
\newblock {\em Math. Proc. Cambridge Philos. Soc. 126}, 3 (1999), 565--574.

\bibitem{2011arXiv1109.1794B}
{\sc {Bergweiler}, W., {Rippon}, P.~J., and {Stallard}, G.~M.}
\newblock {Multiply connected wandering domains of entire functions}.
\newblock {\em To appear in Proc. London Math. Soc., arXiv:1109.1794v1\/}
  (2011).

\bibitem{MR1230383}
{\sc Carleson, L., and Gamelin, T.~W.}
\newblock {\em Complex dynamics}.
\newblock Universitext: Tracts in Mathematics. Springer-Verlag, New York, 1993.

\bibitem{MR816367}
{\sc Douady, A., and Hubbard, J.~H.}
\newblock On the dynamics of polynomial-like mappings.
\newblock {\em Ann. Sci. \'Ecole Norm. Sup. (4) 18}, 2 (1985), 287--343.

\bibitem{MR1102727}
{\sc Eremenko, A.~E.}
\newblock On the iteration of entire functions.
\newblock {\em Dynamical systems and ergodic theory ({W}arsaw 1986) 23\/}
  (1989), 339--345.

\bibitem{MR1771857}
{\sc Ess{\'e}n, M., and Wu, S.}
\newblock Repulsive fixpoints of analytic functions with applications to
  complex dynamics.
\newblock {\em J. London Math. Soc. (2) 62}, 1 (2000), 139--148.

\bibitem{MR1049148}
{\sc Hayman, W.~K.}
\newblock {\em Subharmonic functions. {V}ol. 2}, vol.~20 of {\em London
  Mathematical Society Monographs}.
\newblock Academic Press Inc. [Harcourt Brace Jovanovich Publishers], London,
  1989.

\bibitem{MR0162913}
{\sc Heins, M.}
\newblock {\em Selected topics in the classical theory of functions of a
  complex variable}.
\newblock Athena Series: Selected Topics in Mathematics. Holt, Rinehart and
  Winston, New York, 1962.

\bibitem{MR1642181}
{\sc Herring, M.~E.}
\newblock Mapping properties of {F}atou components.
\newblock {\em Ann. Acad. Sci. Fenn. Math. 23}, 2 (1998), 263--274.

\bibitem{MR2467621}
{\sc Kuczma, M.}
\newblock {\em An introduction to the theory of functional equations and
  inequalities}, second~ed.
\newblock Birkh\"auser Verlag, Basel, 2009.

\bibitem{MR2917092}
{\sc Mihaljevi{\'c}-Brandt, H., and Peter, J.}
\newblock Poincar\'e functions with spiders' webs.
\newblock {\em Proc. Amer. Math. Soc. 140}, 9 (2012), 3193--3205.

\bibitem{MR1334766}
{\sc Ransford, T.}
\newblock {\em Potential theory in the complex plane}, vol.~28 of {\em London
  Mathematical Society Student Texts}.
\newblock Cambridge University Press, Cambridge, 1995.

\bibitem{MR1767705}
{\sc Rippon, P.~J., and Stallard, G.~M.}
\newblock On sets where iterates of a meromorphic function zip towards
  infinity.
\newblock {\em Bull. London Math. Soc. 32}, 5 (2000), 528--536.

\bibitem{pre05533139}
{\sc Rippon, P.~J., and Stallard, G.~M.}
\newblock {Escaping points of entire functions of small growth.}
\newblock {\em Math. Z. 261}, 3 (2009), 557--570.

\bibitem{Rippon01102012}
{\sc Rippon, P.~J., and Stallard, G.~M.}
\newblock Fast escaping points of entire functions.
\newblock {\em Proc. London Math. Soc. (3) 105}, 4 (2012), 787--820.

\bibitem{MR2838342}
{\sc Sixsmith, D.~J.}
\newblock Entire functions for which the escaping set is a spider's web.
\newblock {\em Math. Proc. Cambridge Philos. Soc. 151}, 3 (2011), 551--571.

\bibitem{MR2959920}
{\sc Sixsmith, D.~J.}
\newblock Simply connected fast escaping {F}atou components.
\newblock {\em Pure Appl. Math. Q. 8}, 4 (2012), 1029--1046.

\end{thebibliography}
\end{document}